\documentclass[a4paper]{amsart}

\usepackage[all]{xy}
\usepackage{amsmath}
\usepackage{amssymb}
\usepackage{amsthm}
\usepackage{latexsym}
\usepackage{enumerate}
\usepackage{prettyref}
\usepackage{hyperref}

\newtheorem{mainthm}{Theorem}
\newtheorem{theorem}{Theorem}[section]
\newrefformat{thm}{\hyperref[{#1}]{Theorem~\ref*{#1}}}
\newtheorem{definition}[theorem]{Definition}
\newrefformat{def}{\hyperref[{#1}]{Definition~\ref*{#1}}}

\newrefformat{lem}{\hyperref[{#1}]{Lemma~\ref*{#1}}}
\newtheorem{proposition}[theorem]{Proposition}
\newrefformat{prop}{\hyperref[{#1}]{Proposition~\ref*{#1}}}
\newtheorem{corollary}[theorem]{Corollary}
\newrefformat{cor}{\hyperref[{#1}]{Corollary~\ref*{#1}}}
\newtheorem{remark}[theorem]{Remark}
\newrefformat{rem}{\hyperref[{#1}]{Remark~\ref*{#1}}}
{
\newtheorem{examplecore}[theorem]{Example}}
\newrefformat{ex}{\hyperref[{#1}]{Example~\ref*{#1}}}

\newcommand{\Ao}{\ensuremath{\mathbb{A}^1}}
\newcommand{\hocolim}{\ensuremath{\operatornamewithlimits{hocolim}}}

\newcommand{\Sing}{\ensuremath{\operatorname{Sing}_\bullet^{\Ao}}}

\begin{document}

\title{More Examples of Motivic Cell Structures}   
\author{Matthias Wendt}

\date{November 2010}

\address{Matthias Wendt, Mathematisches Institut,
Uni\-ver\-si\-t\"at Freiburg, Eckerstra\ss{}e 1, 79104,
  Freiburg im Breisgau, Germany}
\email{matthias.wendt@math.uni-freiburg.de}


\begin{abstract}
In this note, we describe motivic cell structures arising from the
Bia{\l}ynicki-Birula decomposition. This provides a description of
stable $\Ao$-homotopy types of smooth projective
$\mathbb{G}_m$-varieties where the $\mathbb{G}_m$-action has isolated
fixed points.  
\end{abstract}

\maketitle
\setcounter{tocdepth}{1}
\tableofcontents

\section{Introduction}

In this note, we describe motivic cell structures for some smooth
projective schemes. The term motivic cell structure was coined in a
paper of Dugger and Isaksen \cite{dugger:isaksen}, as an
$\Ao$-homotopy analogue of CW-structures in algebraic
topology. Motivic cell structures describe how a scheme or a
simplicial sheaf is built out of spheres (whenever that happens to be
possible), and so motivic cell structures can in principle be used to
compute arbitrary generalized cohomology theories. Several examples of
motivic cell structures have been given in \cite{dugger:isaksen}, and
it is the goal of this note to add to the available list.  

For a smooth compact manifold, one can use Morse functions and their
associated gradient flows to obtain a CW-structure on the
manifold. In the algebraic setting, there are no Morse functions
available, any algebraic map $X\rightarrow\Ao$ from a projective
variety to the affine line is constant. However, one can interpret the
gradient flow associated to a Morse function as an
$\mathbb{R}^\times$-action on the manifold. In this formulation, the
existence of CW-structures associated to Morse functions has an
algebraic analogue - a cell structure exists for smooth projective
varieties with a $\mathbb{G}_m$-action with isolated fixed points. 
For these varieties it is already known that the Bia{\l}ynicki-Birula
filtration allows to produce splittings of motives of these varieties,
cf. \cite{brosnan}. We show here that the Bia{\l}ynicki-Birula
filtration actually carries more information, 
which allows to extract a motivic cell structure. For this, however,
we need to assume that the multiplicative group acts with isolated
fixed points. For projective homogeneous varieties under split
reductive groups, we even get unstable cell structures from the
Schubert cell decomposition.  This solves a question posed in
\cite[Remark 4.5]{dugger:isaksen}. 

\begin{mainthm}
\label{main1}
Let $k$ be a field, and let $X$ be a smooth projective variety
equipped with an action of $\mathbb{G}_m$ with isolated fixed
points. Then $X$ is stably cellular. If $X$ is a projective
homogeneous space under a split connected reductive group, then $X$ is
also unstably cellular. 
\end{mainthm}

The motivic cell structures described above immediately induce
splittings of the motives associated to these varieties. In the case
of the smooth projective varieties of Theorem \ref{main1}, the motivic 
splittings are the ones known from
\cite{brosnan}. But the motivic cell structures give complete
descriptions of the stable homotopy types of the considered varieties,
and could in principle be used to compute representable
cohomology theories.  

\emph{Structure of the paper:} In \prettyref{sec:prelims}, we recall
the definition of motivic cell structures and collect a few
preliminary results. In \prettyref{sec:classical}, we produce cell
structures on smooth projective $\mathbb{G}_m$-varieties from the
Bia{\l}ynicki-Birula decomposition. In \prettyref{sec:gcell}, we
deduce cellularity of reductive groups and their classifying spaces. 
Finally, in \prettyref{sec:appl}, we discuss consequences of the
results.  

\section{Preliminaries on motivic cell structures}
\label{sec:prelims}

For the basics of $\Ao$-homotopy theory, we refer to
\cite{morel:voevodsky:1999:a1}. In the sequel, the base scheme will
be a field unless mentioned otherwise.
The $\Ao$-homotopy theory is constructed as the
$\Ao$-localization of the category of simplicial sheaves on the
category $\operatorname{Sm}_k$ of smooth schemes over  the field $k$
equipped with the Nisnevich topology. 

Since we are interested in cell decompositions of varieties,
we next recall the definition of motivic cell structures in the sense
of Dugger and Isaksen, cf. \cite[Definition 2.1]{dugger:isaksen}. 

\begin{definition}
\label{def:cell}
Let $\mathcal{M}$ be a pointed model category and $\mathcal{A}$ be a
set of objects in $\mathcal{M}$. The class of
\emph{$\mathcal{A}$-cellular   objects} is the smallest class of
objects of $\mathcal{M}$ such that  
\begin{enumerate}[(C1)]
\item every object of $\mathcal{A}$ is $\mathcal{A}$-cellular,
\item if $X$ is weakly equivalent to an $\mathcal{A}$-cellular object,
  then $X$ is $\mathcal{A}$-cellular, 
\item if $D:I\rightarrow\mathcal{M}$ is a diagram such that each $D_i$
  is $\mathcal{A}$-cellular, then $\hocolim D$ is
  $\mathcal{A}$-cellular. 
\end{enumerate}
\end{definition}

In the model category for $\Ao$-homotopy theory, we have
spheres 
\begin{displaymath}
S^{p,q}=S^{p-q}_s\wedge\mathbb{G}_m^{\wedge q},
\end{displaymath}
where $S^1_s$ is the simplicial circle. We will abbreviate 
$\{S^{p,q}\mid p,q\in\mathbb{N}\}$-cellular to cellular, as this is
the only notion of cellularity we will use. 

\subsection{Stable cell structures for varieties with cellular
  filtration}  

We show how cellular filtrations on smooth varieties yield homotopy
pushouts. The unusual twist in the proof is that one would usually use 
the closed subvarieties in the filtration to describe the cell
structure. However, to avoid singularity-related problems, it is
better to look at the complements of the subvarieties in the
filtration. A similar trick has also been used by Zibrowius in his work
on computation of Witt groups, cf. \cite{zibrowius}. Recall from
\cite[Definition 3.1]{brosnan} that an 
\emph{affine fibration} is a flat morphism $\phi:X\rightarrow Z$ such
that there exists a Zariski covering $U_i$ of the base $Z$ such that
$\phi|_{U_i}:U_i\times_Z X\cong U_i\times \mathbb{A}^n\rightarrow U_i$
is the projection away from the $\mathbb{A}^n$-factor. 

\begin{proposition}
\label{prop:general}
Let $X$ be a smooth scheme over an arbitrary base scheme $S$. 
Assume that there is a filtration of $X$ by closed subschemes 
\begin{displaymath}
X=X_n\supset X_{n-1}\supset\cdots\supset X_0\supset X_{-1}=\emptyset 
\end{displaymath}
Assume that for each $i$ there is a smooth scheme $Z_i$ and an affine
fibration $\phi_i:X_i\setminus X_{i-1}\rightarrow Z_i$. Then $X$ is
stably $\mathcal{A}$-cellular for 
$$
\mathcal{A}=\{S^{p,q}\wedge\operatorname{Th}(N_i)\mid
p,q\in\mathbb{Z}, 0\leq i< n\}\cup\{Z_n\},
$$
where $\operatorname{Th}(N_i)$ is defined via the homotopy cofibre
sequence 
$$
X\setminus X_i \rightarrow X\setminus X_{i-1}\rightarrow
\operatorname{Th}(N_i). 
$$

In particular, if $S$ is smooth over an excellent Dedekind ring and all
the $X_i\setminus X_{i-1}$ are disjoint unions 
of copies of affine spaces, then $X$ is stably cellular. 
\end{proposition}

\begin{proof}
By assumption, $(X\setminus X_{i-1},X_i\setminus X_{i-1})$ is a smooth
pair, since $Z_i$ is assumed to be smooth, and $X_i\setminus X_{i-1}$
is an affine fibration. 
By the homotopy purity theorem of Morel and Voevodsky,
cf. \cite[Theorem 3.2.23]{morel:voevodsky:1999:a1}, we obtain a
homotopy cofibre sequence
$$
X\setminus X_i\rightarrow X\setminus
X_{i-1}\rightarrow\operatorname{Th}(N_i) .  
$$

Now we use \cite[Lemma 2.5]{dugger:isaksen}, which shows that if in a
cofibre sequence $A\rightarrow B\rightarrow C$ any two spaces are
stably cellular, then so is the third. Applied to the above cofibre 
sequences, we inductively conclude that $X$ is stably cellular. The
base case is $X\setminus X_{n-1}\simeq Z_n$. 

If $X_i\setminus X_{i-1}\cong \bigsqcup_{j}\mathbb{A}^{n_j}$ is a
finite disjoint union of copies of 
affine spaces, then because of the assumption, all normal bundles are
trivial and we find
$$
\operatorname{Th}(N_i)\simeq \bigvee_jS^{2n_j,n_j}.
$$
This settles the final assertion. 
\end{proof}

In particular, it is possible to see the attaching maps from a
cellular filtration, at least stably: the suspension of the attaching
map is $\operatorname{Th}(N_i)\rightarrow \Sigma (X\setminus X_i)$. 

\section{Cell structures from torus actions}
\label{sec:classical}

In this section, we use the Bia{\l}ynicki-Birula method to provide
cell structures for some smooth projective varieties. This method has
already been used to provide decompositions of motives,
cf. \cite{brosnan}. 

\subsection{The Bia{\l}ynicki-Birula decomposition}

The main geometric ingredient for the construction of cell
decompositions of projective homogeneous varieties is the
Bia{\l}ynicki-Birula decomposition. The following formulation can be
found in \cite[Theorem 3.2]{brosnan}. There it is attributed to
Bia{\l}ynicki-Birula, Iversen and Hesselink
\cite{bb1,bb2,hesselink,iversen},  for an explanation of the history,
cf. \cite{brosnan}.  

\begin{theorem}
\label{thm:bbdecomp}
Let $X$ be a smooth, projective variety over a field $k$ equipped with
an action of the multiplicative group $\mathbb{G}_m$. Then
\begin{enumerate}[(i)]
\item The fixed point locus $X^{\mathbb{G}_m}$ is a smooth, closed
  subscheme of $X$. 
\item There is a numbering $X^{\mathbb{G}_m}=\coprod_{i=1}^nZ_i$ of
  the connected components of the fixed point locus, a filtration 
\begin{displaymath}
X=X_n\supset X_{n-1}\supset\cdots\supset X_0\supset X_{-1}=\emptyset 
\end{displaymath}
and affine fibrations $\phi_i:X_i\setminus X_{i-1}\rightarrow Z_i$.  
\item The relative dimension $a_i$ of the affine fibration $\phi_i$ is
  the dimension of the positive eigenspace of the action of
  $\mathbb{G}_m$ on the tangent space of $X$ at an arbitrary point
  $z\in Z_i$. The dimension of $Z_i$ is the dimension of
  $TX_z^{\mathbb{G}_m}$. 
\end{enumerate}
\end{theorem}

\subsection{Stable cell structures for $\mathbb{G}_m$-varieties}  
Now we use the Bia{\l}ynicki-Birula decomposition to obtain stable
cell structures for $\mathbb{G}_m$-varieties with isolated fixed
points. 

\begin{proposition}
\label{prop:bbcell}
Let $X$ be a smooth, projective variety over a field $k$. If there
exists an action of the multiplicative group $\mathbb{G}_m$ on $X$
which has isolated fixed points, then $X$ is stably cellular. 
\end{proposition}

\begin{proof}
We use the Bia{\l}ynicki-Birula decomposition. In case the
$\mathbb{G}_m$-action has only isolated fixed points,  
the fixed locus $X^{\mathbb{G}_m}$ is a union of finitely many
$k$-points $Z_1,\dots,Z_n$. Therefore, there is a filtration 
\begin{displaymath}
X=X_n\supset X_{n-1}\supset\cdots\supset X_0\supset X_{-1}=\emptyset 
\end{displaymath}
and the $X_i\setminus X_{i-1}$ are isomorphic to
affine spaces $\mathbb{A}^{n_i}$. The schemes $X_i\setminus X_{i-1}$
are the sets of points $x\in X$ such that $\lim_{t\rightarrow 0}tx\in
Z_i$, where $x\mapsto tx$ denotes the $\mathbb{G}_m$-action. 
It is then obvious that the open subvariety $X\setminus
X_i\hookrightarrow X\setminus X_{i-1}$ has complement isomorphic to
$\mathbb{A}^{n_i}$. Now we apply \prettyref{prop:general} to this
filtration. 
\end{proof}

Now we will show that there is a great supply of varieties to which
the above result applies. 

\begin{definition}
\label{def:spherical}
Let $k$ be a field, and let $G$ be a connected split reductive
group. A normal $G$-variety $X$ is called \emph{spherical}, if some
Borel subgroup of $G$ has an open orbit on $X$. 
\end{definition}

Particular examples of spherical varieties are projective homogeneous
varieties, toric varieties, symmetric varieties and their wonderful
compactifications.

\begin{proposition}
\label{prop:spherical}
Let $G$ be a split reductive group with Borel subgroup $B$, and let
$X$ be a projective smooth spherical $G$-variety. Then there exists a
group  homomorphism $\mathbb{G}_m\rightarrow B$ such that the induced
$\mathbb{G}_m$-action on $X$ has isolated fixed points.  
\end{proposition}

\begin{proof}
By definition, there are only finitely many $B$-orbits in
$X$. Therefore, the induced action of the maximal torus $T$ has 
finitely many fixed points. We show that there exists a morphism
$\mathbb{G}_m\rightarrow T$ which also has finitely many fixed
points, which basically follows from Thomason's generic slice theorem
\cite[Theorem 4.10]{thomason:1986}. Note that the base scheme $S$ is
the spectrum of a field, $X$ is smooth projective over $k$, so the
conditions are satisfied. The generic slice theorem then states that
there exists an open affine subscheme $U\subseteq X$ and a subtorus
$T'\hookrightarrow T$ such that $U\cong T/T'\times U/T$. Since $T$
acts with isolated fixed points, $T'$ is a strict subtorus -- it is
the generic stabilizer  for $U$. In particular, any subgroup
$\mathbb{G}_m\subseteq T$ with $\mathbb{G}_m\not\subseteq T'$ will act
without fixed points on $U$. Now the complement $X\setminus U$ is also
a $T$-variety, and it has smaller dimension than $X$. Inductively, we
find finitely many proper subtori $T_i\subseteq T$ such that any subgroup
$\mathbb{G}_m\subseteq T$ avoiding these subtori $T_i$ has the same
fixed set as $T$. This concludes the proof.
\end{proof}

\begin{corollary}
Let $G$ be a connected split reductive group over $k$, and let $X$ be
a smooth projective spherical $G$-variety. Then $X$ is stably cellular. 
\end{corollary}

\begin{proof}
By \prettyref{prop:spherical}, there exists a $\mathbb{G}_m$-action with
isolated fixed points on $X$. 
By \prettyref{prop:bbcell}, $X$ is stably cellular.  
\end{proof}

This in particular applies to wonderful compactifications of
semisimple groups. 

\subsection{Unstable cell structures for projective homogeneous
  spaces} 

We have seen in the previous section that projective homogeneous
varieties under split reductive groups are stably cellular. Here, we
use slightly more precise information to obtain unstable cell
structures. Recall that one way to produce the stable cell structures
is the following result which provides a $\mathbb{G}_m$-action with
isolated fixed points, cf. \cite[Section 1]{kim:pandharipande}.

\begin{theorem}
\label{thm:isolated}
Let $G$ be a split semisimple group, let $T$ be a maximal torus, and
let $P$ be a standard parabolic subgroup. Then the following
assertions hold: 
\begin{enumerate}[(i)]
\item The $T$-action on $G/P$ has isolated fixed points. 
\item Let $\mathbb{G}_m\subseteq T$ correspond to an interior point of
  a Weyl-chamber. Then $(G/P)^{\mathbb{G}_m}=(G/P)^T$ and the
  Bia{\l}ynicki-Birula decomposition obtained from the
  $\mathbb{G}_m$-action is an affine stratification of $G/P$. 
\end{enumerate}
\end{theorem}
It is worth noting that the Bia{\l}ynicki-Birula stratification for
the above $\mathbb{G}_m$-action coincides (up to the Weyl group
action) with the Schubert cell stratification. 

We are going to use the Schubert cell stratification to provide the
unstable cells.  

\begin{proposition}
\label{prop:gpcell}
If $X=G/P$ is homogeneous under a split connected reductive group
$G$, then $X$ is unstably cellular. 
\end{proposition}

\begin{proof}
As in \prettyref{prop:bbcell}, there is a filtration 
\begin{displaymath}
X=X_n\supset X_{n-1}\supset\cdots\supset X_0\supset X_{-1}=\emptyset 
\end{displaymath}
and the $X_i\setminus X_{i-1}$ are isomorphic to
affine spaces $\mathbb{A}^{n_i}$. Via the identification with the
Schubert cell stratification, we can describe the cells as cosets
$BwP/P$ with $w$ running through the Weyl group $W$ of $G$. 

The inclusion $X\setminus
X_i\hookrightarrow X\setminus X_{i-1}$ has complement isomorphic to
$\mathbb{A}^{n_i}$ and we fix an isomorphism mapping
$0\in\mathbb{A}^{n_i}$ to the $\mathbb{G}_m$-fixed point $x_i$ of
$X_i$. 

There is a unique $\mathbb{G}_m$-fixed point $z\in X$ of maximal
index, this is the fixed point contained in $X\setminus X_n=Bw_0P/P$
with $w_0$ the longest element of the Weyl group. 

Now there exists $w\in W$ such that $wx_i=z$. This maps the
$\mathbb{A}^{n_i}$ isomorphically to an affine 
subspace of the big cell $X\setminus X_n$. It follows that 
$$
(w^{-1}(X\setminus X_n))\cap X\setminus X_i\cong \mathbb{A}^{\dim
  X}\setminus \mathbb{A}^{n_i}
$$

Then the following is a
Zariski covering of $X\setminus X_{i-1}$:
\begin{center}
  \begin{minipage}[c]{10cm}
    \xymatrix{
      \mathbb{A}^{\dim X}\setminus\mathbb{A}^{n_i} \ar[r] \ar[d] &
      X\setminus X_i \ar[d] \\
      \mathbb{A}^{\dim X}\ar[r] & X\setminus X_{i-1}.
    }
  \end{minipage}
\end{center}
But this implies that we have a homotopy cofibre sequence 
$$\mathbb{A}^{\dim X-n_i}\setminus\{0\}\rightarrow X\setminus
X_i\rightarrow X\setminus X_{i-1}
$$
These cofibre sequences provide the unstable cell decomposition of $X$.
\end{proof}

Note that this result holds over general base schemes as soon as the
split reductive group $G$ is defined over $\mathbb{Z}$. 

\subsection{Example: Even-dimensional Quadrics}

As an example how the above works, we discuss the case of the
$2n$-dimensional projective quadric  $Q_{2n}\subseteq
\mathbb{P}^{2n+1}$ defined by the vanishing of the split symmetric
bilinear form 
\begin{displaymath}
V\left(\sum_{i=0}^n x_iy_i\right)=Q_{2n}\subseteq\mathbb{P}^{2n+1}.
\end{displaymath}
This is a homogeneous space under the split group $PSO(2n+2)$. 
There are two important closed subvarieties, 
$$
Z_x=\{x_0=x_1=\cdots=x_n=0\}, \qquad
Z_y=\{y_0=y_1=\cdots=y_n=0\},
$$
which arise from the Bia{\l}ynicki-Birula decomposition associated to 
the $\mathbb{G}_m$-action 
$$
\mathbb{G}_m\times Q_{2n}\rightarrow Q_{2n}:
(t,[x_0:\dots:x_n:y_0:\cdots:y_n])\mapsto
[tx_0:\cdots:tx_n:y_0:\cdots:y_n]. 
$$

We have $Z_x\cong Z_y\cong\mathbb{P}^n$. The obvious projection 
$\pi:Q_{2n}\setminus Z_x\rightarrow Z_y$ onto the $x$-coordinates is a
rank $n$ vector bundle. Therefore, $Q_{2n}\setminus Z_x\simeq
\mathbb{P}^n$ for which we have the usual cell structure.

Now $Z_x\cong\mathbb{P}^n$ has a filtration by projective spaces
$\mathbb{P}^k$, $0\leq k\leq n$, which are given by the equations 
$$
x_{k+1}=x_{k+2}=\cdots=x_n=0.$$
We denote by $V_k$ the subvariety of $Q_{2n}$ given by $x_k\neq 0$,
which is isomorphic to $\mathbb{A}^{2n}$. 
The intersection $\mathbb{P}^k\cap V_k$ is then isomorphic to affine 
$k$-dimensional subspace $\mathbb{A}^k$. Thus we have a Zariski
covering 
\begin{center}
  \begin{minipage}[c]{10cm}
    \xymatrix{
      \mathbb{A}^{2n-k}\setminus\{0\}\simeq
      \mathbb{A}^{2n}\setminus\mathbb{A}^k\ar[r] \ar[d]&  
      Q_{2n}\setminus\mathbb{P}^{k} \ar[d] \\ 
      \mathbb{A}^{2n}\cong V_k \ar[r] & Q_{2n}\setminus\mathbb{P}^{k-1}.
    }
  \end{minipage}
\end{center}
These yield homotopy cofibre sequences
$$
\mathbb{A}^{2n-k}\setminus\{0\}\rightarrow
Q_{2n}\setminus\mathbb{P}^{k}\rightarrow
Q_{2n}\setminus\mathbb{P}^{k-1}
$$
which explain how to successively attach cells to $Q_{2n}\setminus
Z_x$ to finally obtain $Q_{2n}$. 
This recovers the classical cell decomposition known for the
projective even-dimensional quadrics. 

\subsection{Remarks on the general case}

Finally, we want to discuss the case of arbitrary smooth projective
$\mathbb{G}_m$-varieties. As formulated in \prettyref{thm:bbdecomp},
the fixed point locus $X^{\mathbb{G}_m}$ is a smooth closed
subscheme, and the successive strata $X_i\setminus X_{i-1}$ are vector
bundle torsors over components of $X^{\mathbb{G}_m}$. We can then
apply \prettyref{prop:general}  to see that only Thom spaces over
components of the fixed point locus $X^{\mathbb{G}_m}$ are needed in
order to reconstruct the stable homotopy type of $X$. 

In particular, we get the following refinement of the  motivic
decompositions for isotropic projective homogeneous varieties given by
Chernousov, Gille and Merkurjev \cite{chernousov:gille:merkurjev}
resp. by Brosnan \cite{brosnan}. 

\begin{proposition}
Let $k$ be a field, and let $X$ be an isotropic projective homogeneous
variety under a reductive group $G$. Then the Bia{\l}ynicki-Birula
filtration   
\begin{displaymath}
X=X_n\supset X_{n-1}\supset\cdots\supset X_0\supset X_{-1}=\emptyset 
\end{displaymath}
induces homotopy cofibre sequences 
$$
X\setminus X_i \rightarrow X\setminus X_{i-1}\rightarrow
\operatorname{Th}(Z)
$$
where $Z$ is a quasi-homogeneous projective variety under the
anisotropic kernel of $G$.
\end{proposition}

\begin{proof}
This follows from the geometric analysis of the Bia{\l}ynicki-Birula
filtration in \cite[Section 4]{brosnan} and the argument in
\prettyref{prop:bbcell}. 
\end{proof}

This homotopy colimit description of isotropic projective homogeneous
varieties in the stable homotopy category refines the motivic
decomposition known from \cite{chernousov:gille:merkurjev} and
\cite{brosnan}. We see that to describe the stable homotopy types of
isotropic projective homogeneous spaces, we need stable homotopy types
of Thom spaces over anisotropic projective homogeneous varieties. It
is not clear to me if the normal bundles of anisotropic varieties are
trivial, i.e. if there are $\Ao$-local weak equivalences
$\operatorname{Th}(Z)\simeq S^{2n,n}\wedge Z_+$.

\section{Cellularity for groups and classifying spaces}
\label{sec:gcell}

The cell structure for projective homogeneous varieties can be lifted
to show cellularity for split reductive groups. The following was also
proved in \cite{morel:fmconj}.

\begin{proposition}
\label{prop:gcell}
Let $G$ be a connected split reductive group. Then $G$ is
stably cellular. 
\end{proposition}

\begin{proof}
Let $B$ be a Borel subgroup of $G$. Then there is a smooth morphism
$\pi:G\rightarrow G/B$. Now let 
\begin{displaymath}
G/B=X_n\supset X_{n-1}\supset\cdots\supset X_0\supset X_{-1}=\emptyset 
\end{displaymath}
be the filtration yielding the cell structure of $G/B$ from
\prettyref{prop:gpcell}. We obtain a filtration 
\begin{displaymath}
G=Y_n\supset Y_{n-1}\supset\cdots\supset Y_0\supset Y_{-1}=\emptyset 
\end{displaymath}
by setting $Y_i=\pi^{-1}(X_i)$. As in the proof of
\prettyref{prop:gpcell}, we obtain a Zariski covering 
\begin{center}
  \begin{minipage}[c]{10cm}
    \xymatrix{
      \pi^{-1}(\mathbb{A}^{\dim X}\setminus \mathbb{A}^{n_i})\cong
      \mathbb{A}^{\dim X}\setminus \mathbb{A}^{n_i}\times B\ar[r]
      \ar[d] &       Y\setminus Y_i \ar[d] \\
      \pi^{-1}(\mathbb{A}^{\dim X}) \cong \mathbb{A}^{\dim X}\times
      B\ar[r] & Y\setminus Y_{i-1}.  
    }
  \end{minipage}
\end{center}
The spaces $\pi^{-1}(\mathbb{A}^{\dim X}\setminus \mathbb{A}^{n_i})$
and $\pi^{-1}(\mathbb{A}^{\dim X})$ are $B$-bundles over
$\mathbb{A}^{\dim X}\setminus \mathbb{A}^{n_i}$ and $\mathbb{A}^{\dim
  X}$ respectively. The $B$-bundle $\pi^{-1}(\mathbb{A}^{\dim X})$ can
be factored as composition of a $T$-bundle followed by a vector
bundle. Therefore, $\pi^{-1}(\mathbb{A}^{\dim X})\rightarrow
\mathbb{A}^{\dim X}$ is a trivial $B$-bundle. The bundle
$\pi^{-1}(\mathbb{A}^{\dim X}\setminus \mathbb{A}^{n_i}) \rightarrow
\mathbb{A}^{\dim X}\setminus \mathbb{A}^{n_i}$ is the restriction of
the former bundle along the inclusion $\mathbb{A}^{\dim X}\setminus
\mathbb{A}^{n_i}\hookrightarrow \mathbb{A}^{\dim X}$, hence also
trivial. This justifies the isomorphisms in the above diagram. 

Stable cellularity now follows by induction, and the fact that stable
cellularity is preserved under products, cf. \cite[Lemma
3.4]{dugger:isaksen}. 
\end{proof}

This approach can be used to reproduce the motivic decompositions of
split reductive groups obtained by Biglari \cite{biglari:diss}. 

From the cellularity of the group $G$ we can deduce the cellularity of
the classifying space $BG$. Here we consider the classifying space of
Nisnevich locally trivial torsors. 

\begin{proposition}
Let $G$ be a connected split reductive group over an infinite field
$k$. Then $BG$ is stably cellular. 
\end{proposition}

\begin{proof}
We use the proof of \cite[Theorem 2.D.11]{farjoun:1996:cellular}. We
start with the fibre sequence $\Sing G\rightarrow E\Sing G\rightarrow
B\Sing G$. This is an $\Ao$-local fibre sequence, cf. \cite[Theorem
4.7]{torsors}. 

We define a sequence of fibrations $F_i\rightarrow E_i\rightarrow
B\Sing G$ by setting
$$
E_0=E\Sing G, F_0=\Sing G, E_{i+1}=E_i\cup CF_i,
$$
and $F_{i+1}$ is the homotopy fibre of the obvious morphism
$E_{i+1}\rightarrow B\Sing G$ in the $\Ao$-local category. By Ganea's
theorem \cite[Proposition 2.22]{classify}, we have in the simplicial
model structure 
$$
F_{i+1}\simeq F_i\ast \Sing G\simeq \Sigma (F_i\wedge \Sing G).
$$
Inductively, we conclude that the simplicial connectivity of $F_{i+1}$
is at least $i$. By Morel's $\Ao$-connectivity theorem \cite[Theorem
15]{morel:2006:a1algtop}, the space $F_{i+1}$ is $i$-$\Ao$-connected. 
Since $B\Sing G$ is $\Ao$-local, the homotopy fibre of
$E_{i+1}\rightarrow B\Sing G$ is $\Ao$-weakly equivalent to $F_{i+1}$,
this is a consequence of \cite[Theorem 1]{flocal}. 

From the assumption and \prettyref{prop:gcell}, it follows that
$\Sing G$ is stably cellular, and inductively, we conclude that $E_i$ and
$F_i$ are stably cellular. This follows since $E_{i+1}$ is constructed
as homotopy cofibre of $F_i\rightarrow E_i$, and as seen above,
$F_{i+1}$ is the suspension of a smash-product of $F_i$ and $\Sing
G$. 

We denote by $F_\infty\rightarrow E_\infty\rightarrow B\Sing G$ the
colimit of the fibre sequences above. This is a simplicial fibre
sequence, by homotopy distributivity, cf. \cite[Proposition
2.17]{classify}. Since $B\Sing G$ is $\Ao$-local, the homotopy fibre
of $E_\infty\rightarrow B\Sing G$ is $\Ao$-weakly equivalent to
$F_\infty$. But as noted above, $F_\infty$ is simplicially
contractible. Therefore, $B\Sing G\simeq \operatorname{hocolim} E_i$,
which implies stable cellularity of $B\Sing G$. 
\end{proof}

\begin{remark}
Alternatively, the construction of the classifying space shows
directly that $BG$ is $G$-cellular in the simplicial model
structure. This implies that $BG$ is also $G$-cellular in the
$\Ao$-local model structure. But \prettyref{prop:gcell} implies that
$\Ao$-locally, $G$ is cellular, so $BG$ is cellular.
\end{remark}

\section{Consequences} 
\label{sec:appl}

\subsection{Motivic Decompositions}

As a consequence of the motivic cell structures developed, we obtain
stable homotopy theory proofs of the motivic decompositions in
\cite{brosnan}. In this section, motives are objects in Voevodsky's
category $DM(k)$. The functor $\operatorname{Sm}_k\rightarrow DM_k$
extends to a functor $\mathcal{SH}(k)\rightarrow DM(k)$, so in fact we
can consider motives associated to arbitrary stable homotopy types. 

\begin{corollary}
\label{cor:motive}
Let $X$ be a smooth projective variety over a field $k$. Assume there
is a $\mathbb{G}_m$-action on $X$ with isolated fixed points. Then the
motive of $X$ splits as a direct sum of Tate motives:
$$
M(X)\cong \bigoplus_i \mathbb{Z}(n_i)[2n_i].
$$
\end{corollary}

\begin{proof}
From \prettyref{prop:bbcell}, we obtained homotopy cofibre sequences 
$$
X\setminus X_i\rightarrow X\setminus X_{i-1}\rightarrow
S^{2(\dim X-n_i),\dim
  X-n_i} 
$$
describing the stable $\Ao$-homotopy type of $X$. These induce
distinguished triangles in $\mathcal{SH}(k)$, and therefore also
distinguished triangles in $DM(k)$. Now we prove the result
inductively for the motives $M(X\setminus X_i)$. For convenience, we
rewrite the above distinguished triangles as 
$$
M(S^{2(\dim X-n_i)-1,\dim X-n_i})\rightarrow M(X\setminus X_i)\rightarrow
M(X\setminus X_{i-1}).
$$
By inductive assumption $M(X\setminus X_i)$ is a direct sum of Tate
motives of the form $\mathbb{Z}(n_j)[2n_j]$. It then suffices to show
that the morphism 
$$
M(S^{2(\dim X-n_i)-1,\dim X-n_i})\rightarrow M(X\setminus X_i)
$$
is trivial, hence induces a splitting 
$$
M(X\setminus X_{i-1})\cong M(S^{2(\dim X-n_i),\dim X-n_i})\oplus
M(X\setminus X_i).
$$

Recall that in the Bia{\l}ynicki-Birula filtration, the dimension of
the big cell in the step $X_i$ equals the dimension of $X_i$, so for
$i>j$, we have $n_i>n_j$. So all the Tate motives in $M(X\setminus
X_i)$ have weight at most $\dim X-n_i$. So the composition of the
morphism 
$$
M(S^{2(\dim X-n_i)-1,\dim X-n_i})\rightarrow M(X\setminus X_i)
$$
with a projection onto a summand of $M(X\setminus X_i)$ is of the form
$\mathbb{Z}(m)[2m-1]\rightarrow \mathbb{Z}(n)[2n]$ with $m\geq
n$. This is always $0$. 
\end{proof}

For projective homogeneous varieties, one can obtain the various
weights of the summands from the Hasse diagram,
cf. \cite{semenov,brosnan}. 

\subsection{Morel's Decomposition Theorem}

We can also obtain decomposition theorems for projective homogeneous
varieties in the $\Ao$-stable homotopy theory. For this, we recall the
following splitting of the $\Ao$-sphere spectrum,
cf. \cite{morel:2006:splitting}: 
Denoting by $\epsilon\in [S^0,S^0]$ the automorphism induced by the
permutation of factors in $\mathbb{G}_m\wedge\mathbb{G}_m$, there are
orthogonal idempotents 
\begin{displaymath}
e_-=\frac{\epsilon+1}{2}, \qquad e_+=\frac{\epsilon-1}{2}
\end{displaymath}
acting on $S^0[\frac{1}{2}]$. These orthogonal idempotents
yield a splitting 
$$SH(S)_\mathbb{Q}\cong SH(S)_{\mathbb{Q}_+}\times
SH(S)_{\mathbb{Q}_-}. 
$$

The motivic cell structures developed in this note allow to produce
descriptions of $\mathbb{Q}_+$ and $\mathbb{Q}_-$-local $\Ao$-stable
homotopy types of projective homogeneous varieties.

The same proof as for the motivic decompositions imply
$\mathbb{Q}_+$-decompositions provided that there are no morphisms
$S^{2n-1,n}\rightarrow S^{2k,k}$ for $k<n$. 

\begin{proposition}
\label{prop:qplus}
Let $X$ be a smooth projective variety over a field $k$. Assume there
is a $\mathbb{G}_m$-action on $X$ with isolated fixed points. 
Assume furthermore that we have 
$$
\operatorname{Hom}_{\mathcal{SH}(k)\otimes\mathbb{Q}_+}(S^{2i-1,i},S^{2j,j})=0
$$
for all $i>j+1$. Then the $\mathbb{Q}_+$-localization of $X$ splits as a
wedge of $(2n,n)$-spheres 
$$
X\simeq_{\mathbb{Q}_+} \bigvee_i S^{2n_i,n_i}.
$$
\end{proposition}

\begin{proof}
By \cite[Corollary 3.43]{morel:2006:a1algtop}, we have 
$$
\operatorname{Hom}_{\mathcal{SH}(k)}(S^{2i-1,i},S^{2j,j})=0
$$
for all $i\leq j$ and
$$
\operatorname{Hom}_{\mathcal{SH}(k)}(S^{2i-1,i},S^{2j-2,j-1})=K^{MW}_{-1}(k).
$$
In the latter group, every
element is a multiple of the Hopf map, i.e. the morphism
$K^{MW}_0(k)\rightarrow K^{MW}_{-1}(k)$ induced by
multiplication by $\eta$ is surjective. But since
$\epsilon\eta=\eta$ in $K^{MW}$, we find $e_+\eta=0$, so the Hopf map
is annihilated in the $\mathbb{Q}_+$-localization. Therefore all
morphisms $S^{2i-1,i}\rightarrow S^{2i-2,i-1}$ are
$\mathbb{Q}_+$-null.  By the assumption above, we then have 
$$
\operatorname{Hom}_{\mathcal{SH}(k)\otimes\mathbb{Q}_+}(S^{2i-1,i},S^{2j,j})=0
$$
for all $i$ and $j$. 

Now the same proof as for \prettyref{cor:motive} yields the claim.
\end{proof}

In particular, under the assumptions of the proposition, we can
decompose a projective homogeneous variety $G/P$ in the
$\mathbb{Q}_+$-localization into a wedge of $S^{2n,n}$-spheres indexed
by the vertices of the Hasse diagram.

For the $\mathbb{Q}_-$-localization, we get the stable homotopy types
of the real realization. This actually works in 
$\mathcal{SH}(k)\otimes\mathbb{Z}[\frac{1}{2}]_-$.  

\begin{proposition}
Let $k$ be a field, and let $X$ be a homotopy colimit of
$S^{p,q}$-spheres. Then $X$ is already $\{S^{n,0}\mid
n\in\mathbb{Z}\}$-cellular in the $\mathbb{Z}[\frac{1}{2}]_-$-local
stable homotopy category. In particular, if $k$ has a real embedding
$k\hookrightarrow \mathbb{R}$, there is a
$\mathbb{Z}[\frac{1}{2}]_-$-local weak equivalence $X\simeq
X(\mathbb{R})$, where $X(\mathbb{R})$ is the (locally constant)
simplicial sheaf on $\operatorname{Sm}_k$ associated to the simplicial
set $X(\mathbb{R})$. 
\end{proposition}

\begin{proof}
The assertion follows immediately from the fact that $S^0\simeq
\mathbb{G}_m$ in $\mathcal{SH}(k)\otimes\mathbb{Z}[\frac{1}{2}]_-$,
cf. \cite[Remark 1.6]{morel:2006:splitting}.
\end{proof}

The $\mathbb{Z}[\frac{1}{2}]_-$-local stable homotopy category can
then serve as a replacement of real realization of cellular varieties
over fields which do not have real embeddings. 
This could be used to prove indecomposability of stable homotopy
types even in situations where no real or complex realization is
available. 

\subsection{$\Ao$-homology and rigidity}

We shortly recall from \cite{morel:2006:a1algtop}  the definition of
$\Ao$-homology: to a simplicial sheaf $X$ one can associate
corresponding sheaf of chain complexes. The sheaves of homology groups
of the corresponding fibrant replacement in the category of Nisnevich
sheaves of abelian groups on $\operatorname{Sm}_k$ are called
$\Ao$-homology sheaves and denoted by
$\mathbf{H}_\bullet^{\Ao}(X,\mathbb{Z})$.  From the $\Ao$-derived
category of Nisnevich sheaves of abelian groups on
$\operatorname{Sm}_k$ one can pass to the $\mathbb{P}^1$-stable
$\Ao$-derived category by formally inverting the ``Tate object''
$\tilde{C_\ast^{\Ao}}(\mathbb{P}^1)[-2]$. For more details on the
definition cf. \cite{asok:haesemeyer}. The cohomology in the
$\mathbb{P}^1$-stable $\Ao$-derived category is denoted by
$\mathbf{H}_\bullet^{s\Ao}(X,\mathbb{Z})$. The latter are sheaves
called the $\mathbb{P}^1$-stable $\Ao$-homology sheaves.

Note that the homotopy pushouts produced from the cellular filtrations
induce corresponding long exact sequences in $\Ao$-homology. The
motivic cell structures therefore can be used to compute
$\Ao$-homology of smooth projective spherical varieties in terms of
$\Ao$-homology of the spheres $S^{2n,n}$. The $\mathbb{G}_m$-stable
$\Ao$-homology has the advantage that it has suspension isomorphisms
not just for suspension with $S^{n,0}$, but also for suspension with
$S^{2n,n}$. 

The $\mathbb{P}^1$-stable $\Ao$-homology of a variety with
motivic cell structure can then be computed from the cell
structure: the homotopy cofibre sequences $X\setminus X_i\rightarrow
X\setminus X_{i-1}\rightarrow\operatorname{Th}(N_i)$ from
\prettyref{prop:general} induce long exact sequences of $\Ao$-homology
$$
\cdots\rightarrow \mathbf{H}_\bullet^{\Ao}(X\setminus X_i,\mathbb{Z}) 
\rightarrow \mathbf{H}_\bullet^{\Ao}(X\setminus X_{i-1},\mathbb{Z}) 
\rightarrow \mathbf{H}_\bullet^{\Ao}(\operatorname{Th}(N_i),\mathbb{Z})
\rightarrow \cdots
$$
If the complements in the cellular filtration are unions of affine
spaces, then $\operatorname{Th}(N_i)\simeq \bigvee_j S^{2n_j,n_j}$,
and we get a long exact sequence (in \emph{$\mathbb{P}^1$-stable}
$\Ao$-homology): 
$$
\cdots\rightarrow \mathbf{H}_\bullet^{s\Ao}(X\setminus X_i,\mathbb{Z}) 
\rightarrow \mathbf{H}_\bullet^{s\Ao}(X\setminus X_{i-1},\mathbb{Z}) 
\rightarrow
\bigoplus_j\mathbf{H}_{\bullet+n_j}^{s\Ao}(\operatorname{Spec}(k),\mathbb{Z}) 
\rightarrow \cdots
$$
Informally, one would like to state the result as a module isomorphism 
$$
\mathbf{H}_\bullet^{s\Ao}(X,\mathbb{Z})\cong H_\bullet(X(\mathbb{C}),\mathbb{Z})
\otimes_{\mathbb{Z}}\mathbf{H}_\bullet^{s\Ao}(\operatorname{Spec}k,\mathbb{Z}).
$$
This is, however, not accurate, since the attaching maps in
$\Ao$-homology have much more information, and the passage to the
complex realization loses a lot of this information. Nevertheless, for
varieties with a motivic cell structure, one can rather easily
evaluate the $\mathbb{P}^1$-stable $\Ao$-homology, provided one uses
the $\mathbb{P}^1$-stable $\Ao$-homology of the base field as a black
box. Note that similar arguments hold for reductive
groups and their classifying spaces by \prettyref{sec:gcell}. 

We finally want to note the following: 
\begin{corollary}
Let $X$ be a smooth projective variety over $k$ with a
$\mathbb{G}_m$-action with isolated fixed points. Then the $\Ao$-chain
complex of $X$ is of mixed Tate type in the sense of \cite[Definition
4.7]{morel:fmconj}. 
\end{corollary}

In light of \cite[Theorem 12]{morel:fmconj}, it is reasonable to
conjecture that rigidity also holds for $\Ao$-homotopy of the
varieties in the above corollary, even though they are not
$\Ao$-simply-connected. This would imply that the $\Ao$-homotopy and
$\Ao$-homology of these varieties over algebraically closed fields
agrees with the homotopy and homology of the complex realizations, at
least with finite coefficients away from the characteristic.


\begin{thebibliography}{Wen10b}

\bibitem[AH11]{asok:haesemeyer}
A. Asok and C. Haesemeyer. The $0$-th stable $\Ao$-homotopy sheaf and
quadratic zero cycles. arXiv:1108.3854.

\bibitem[BB73]{bb1}
A. Bia{\l}ynicki-Birula. 
Some theorems on actions of algebraic groups.
Ann. of Math. (2) 98 (1973), 480--497. 

\bibitem[BB76]{bb2}
A. Bia{\l}ynicki-Birula. 
Some properties of the decompositions of algebraic varieties
determined by actions of a
torus. Bull. Acad. Polon. Sci. S{\'e}r. Sci. Math. Astronom. Phys. 24
(1976), no. 9, 667--674.  

\bibitem[Big05]{biglari:diss}
S. Biglari. Motives of reductive groups. PhD thesis, University of
Leipzig, 2005.

\bibitem[Bri97]{brion}
M. Brion. Equivariant Chow groups for torus actions. 
Transform. Groups 2 (1997), no. 3, 225--267. 

\bibitem[Bro05]{brosnan}
P. Brosnan. 
On motivic decompositions arising from the method of
Bia{\l}ynicki-Birula. Invent. Math. 161 (2005), no. 1, 91--111.  

\bibitem[CGM05]{chernousov:gille:merkurjev}
V. Chernousov and S. Gille and A. Merkurjev. 
Motivic decomposition of isotropic projective homogeneous varieties.  
Duke Math. J. 126 (2005), no. 1, 137--159. 

\bibitem[DF96]{farjoun:1996:cellular}
E. Dror Farjoun. Cellular spaces, null spaces and homotopy
localization. Lecture Notes in Mathematics 1622, Springer (1996).

\bibitem[DI05]{dugger:isaksen}
D. Dugger and D.C. Isaksen. Motivic cell structures. 
Algebr. Geom. Topol. 5 (2005), 615--652 (electronic). 



\bibitem[Hes81]{hesselink}
W.H. Hesselink. 
Concentration under actions of algebraic groups. In: Paul Dubreil and 
Marie-Paule Malliavin Algebra Seminar, 33rd Year (Paris, 1980),
pp. 55--89, 
Lecture Notes in Math., 867, Springer, Berlin, 1981. 

\bibitem[Ive72]{iversen}
B. Iversen. 
A fixed point formula for action of tori on algebraic varieties.
Invent. Math. 16 (1972), 229--236. 

\bibitem[KP00]{kim:pandharipande}
B. Kim and R. Pandharipande. 
The connectedness of the moduli space of maps to homogeneous
spaces. In: Symplectic geometry and mirror symmetry
(Seoul, 2000), 187--201, World Sci. Publ., River Edge, NJ, 2001.  


\bibitem[Mor06a]{morel:2006:splitting}
F. Morel. Rational stable splitting of Grassmannians and rational
motivic sphere spectrum. Preprint (2006). 

\bibitem[Mor12a]{morel:2006:a1algtop}
F. Morel. $\Ao$-algebraic topology over a field. Lecture notes in
mathematics 2052, Springer, 2012.

\bibitem[Mor12b]{morel:fmconj}
F. Morel On the Friedlander-Milnor conjecture for groups of small rank.
In: Current Developments in Mathematics (2010). International Press,
2012. 

\bibitem[MV99]{morel:voevodsky:1999:a1}
F. Morel and V. Voevodsky. $\Ao$-homotopy theory of
schemes. Publ. Math. Inst. Hautes \'Etudes Sci. 90 (1999),
45--143. 

\bibitem[Sem06]{semenov}
N. Semenov. Motives of projective homogeneous spaces. PhD thesis,
Ludwig-Maximilians-Universit\"at M\"unchen, 2006. 

\bibitem[Tho86]{thomason:1986}
R.W. Thomason. Comparison of equivariant algebraic and topological
$K$-theory. Duke Math. J. 53 (1986), no. 3, 795--825. 


\bibitem[Wen11a]{torsors}
M. Wendt. Rationally trivial torsors in $\Ao$-homotopy
theory. J. K-Theory 7 (2011), no. 3, 541--572. 

\bibitem[Wen11b]{classify}
M. Wendt. Classifying spaces and fibrations of simplicial
sheaves. J. Homotopy Relat. Struct. 6 (2011), no. 1, 1--38. 

\bibitem[Wen10]{flocal}
M. Wendt. Fibre sequences and localization of simplicial
sheaves. Preprint, 2010.  arXiv:1011:4784.

\bibitem[Zib10]{zibrowius}
M. Zibrowius. Witt groups of complex cellular varieties.
Doc. Math. 16 (2011), 465--511. 


\end{thebibliography}
\end{document}